\documentclass[a4paper,12pt]{amsart}
\usepackage{times} 
\usepackage{latexsym,amscd,amssymb,url}
\usepackage[latin1]{inputenc} 	
\usepackage[T1]{fontenc}         
\usepackage{ae}									
\usepackage[all,cmtip]{xy}  
\usepackage{graphicx}
\usepackage{MnSymbol}

\newtheorem{theorem}{Theorem}

\theoremstyle{plain}

\newtheorem{conjecture}[theorem]{Conjecture}
\newtheorem{corollary}[theorem]{Corollary}

\newtheorem{definition}[theorem]{Definition}

\newtheorem{lemma}[theorem]{Lemma}

\newtheorem{proposition}[theorem]{Proposition}
\newtheorem{remark}[theorem]{Remark}

\newcommand{\Z}{\mathbb Z}

\newcommand{\C}{\mathbb C}

\newcommand{\R}{\operatorname{\textbf{R}}}

\newcommand{\CP}{\mathbb P}

\newcommand{\D}{\operatorname{D}}

\newcommand{\Pic}{\operatorname{Pic}}

\newcommand{\pr}{\operatorname{pr}}

\newcommand{\Sing}{\operatorname{Sing}}
\newcommand{\Alb}{\operatorname{Alb}}
\newcommand{\PD}{\operatorname{PD}}

\newcommand{\T}{\mathcal V}

\newcommand{\AJ}{\operatorname{AJ}}

\newcommand{\reduced}{\operatorname{red}}
\newcommand{\Wdiv}{\operatorname{div}}
\newcommand{\inv}{\operatorname{inv}}
\newcommand{\mov}{\operatorname{mov}}
\newcommand{\supp}{\operatorname{supp}}
\newcommand{\Gr}{\operatorname{Gr}}

\newcommand{\dashedlongrightarrow}{\xymatrix@1@=15pt{\ar@{-->}[r]&}}
\renewcommand{\longrightarrow}{\xymatrix@1@=15pt{\ar[r]&}}
\renewcommand{\mapsto}{\xymatrix@1@=15pt{\ar@{|->}[r]&}}
\renewcommand{\twoheadrightarrow}{\xymatrix@1@=15pt{\ar@{->>}[r]&}}
\newcommand{\hooklongrightarrow}{\xymatrix@1@=15pt{\ar@{^(->}[r]&}}
\newcommand{\congpf}{\xymatrix@1@=15pt{\ar[r]^-\sim&}}
\renewcommand{\cong}{\simeq}

\begin{document}

\title[Theta divisors with curve summands and the Schottky Problem]{Theta divisors with curve summands and the Schottky Problem}

\author{Stefan Schreieder} 
\address{Max-Planck-Institut für Mathematik, Vivatsgasse 7, 53111 Bonn, Germany} 
\address{Mathematisches Institut, Universit\"at Bonn, Endenicher Allee 60, 53115 Bonn, Germany} 
\email{schreied@math.uni-bonn.de} 

\date{August 25, 2015; \copyright{\ Stefan Schreieder 2015}}
%
%


\begin{abstract}
We prove the following converse of Riemann's Theorem:
let $(A,\Theta)$ be an indecomposable principally polarized abelian variety whose theta divisor can be written as a sum of a curve and a codimension two subvariety $\Theta=C+Y$.
Then $C$ is smooth, $A$ is the Jacobian of $C$, and $Y$ is a translate of $W_{g-2}(C)$. 
As applications, we determine all theta divisors that are dominated by a product of curves and characterize Jacobians by the existence of a $d$-dimensional subvariety with curve summand whose twisted ideal sheaf is a generic vanishing sheaf.
\end{abstract}

\maketitle

\section{Introduction}
This paper provides new geometric characterizations of Jacobians inside the moduli stack of all principally polarized abelian varieties 
over the complex numbers. 
For a recent survey on existing solutions and open questions on the Schottky Problem, we refer the reader to \cite{grushevsky}. 

By slight abuse of notation, we denote a ppav (principally polarized abelian variety) by $(A,\Theta)$, where $\Theta\subseteq A$ is a theta divisor that induces the principal polarization on the abelian variety $A$; the principal polarization determines $\Theta\subseteq A$ uniquely up to translation.

\subsection{A converse of Riemann's theorem}
Let $(J(C),\Theta_C)$ be the Jacobian of a smooth curve $C$ of genus $g\geq 2$. 
We fix a base point on $C$ and consider the corresponding Abel--Jacobi embedding $C\longrightarrow J(C)$. 
Addition of points induces morphisms
\[
\AJ_{k}: C^{(k)}\longrightarrow J(C) ,
\] 
whose image is denoted by $W_{k}(C)$.  
Riemann's Theorem \cite[p.\ 27]{arbarello-etal} says $\Theta_C=W_{g-1}(C)$.
That is, if we identify $C$ with its Abel--Jacobi image $W_1(C)$, then $\Theta_C$ can be written as a $(g-1)$-fold sum 
$
\Theta_C=C+ \dots + C . 
$ 
We prove the following converse. 

\begin{theorem} \label{thm:Theta}
Let $(A,\Theta)$ be an indecomposable $g$-dimensional ppav.  
Suppose that there is a curve $C$ and a codimension two subvariety $Y$ in $A$ such that
\begin{align*} 
\Theta=C+Y .
\end{align*}
Then $C$ is smooth and there is an isomorphism $(A,\Theta)\cong (J(C),\Theta_C)$ which identifies $C$ and $Y$ with translates of $W_1(C)$ and $W_{g-2}(C)$, respectively.
\end{theorem}
 

The intermediate Jacobian of a smooth cubic threefold is an indecomposable ppav which is not isomorphic to the Jacobian of a curve and whose theta divisor can be written as a sum of two surfaces \cite[Sec.\ 13]{clemens-griffiths}.
The analogue of Theorem \ref{thm:Theta} is therefore false if one replaces $C$ and $Y$ by subvarieties of arbitrary dimensions. 

Recall that a $d$-dimensional subvariety $Z\subseteq A$ is called geometrically non-degenerate if there is no nonzero decomposable holomorphic $d$-form on $A$ which restricts to zero on $Z$, see \cite[p.\ 466]{ran}. 
One of Pareschi--Popa's conjectures (Conjecture \ref{conj:pp} below) predicts that apart from Jacobians of curves, intermediate Jacobians of smooth cubic threefolds are the only ppavs whose theta divisors have a geometrically non-degenerate summand of dimension $1\leq d \leq g-2$. 
Theorem \ref{thm:Theta} proves (a strengthening of) that conjecture if $d=1$ or $d=g-2$.

\subsection{Detecting Jacobians via special subvarieties} 
Recall that a coherent sheaf $\mathcal F$ on an abelian variety $A$ is a GV-sheaf if for all $i$ its $i$-th cohomological support locus
\[
S^i(\mathcal F):=\left\{L\in \Pic^0(A)\mid H^i(A,\mathcal F\otimes L)\neq 0\right\} 
\] 
has codimension $\geq i$ in $\Pic^0(A)$, see \cite[p.\ 212]{pareschi-popa}. 

Using this definition, we characterize
$
W_d(C)\subseteq J(C)
$
among all $d$-dimensional subvarieties of arbitrary ppavs.  
Our proof combines Theorem \ref{thm:Theta} with the main results in \cite{debarre} and \cite{pareschi-popa}.

\begin{theorem} \label{thm:GV}
Let $(A,\Theta)$ be an indecomposable ppav, and let $Z\subsetneq A$ be a geometrically non-degenerate subvariety of dimension $d$.
Suppose that the following holds: 
\begin{enumerate}
	\item $Z=C+Y$ has a curve summand $C\subseteq A$, \label{item:Y=Y'+C}
	\item the twisted ideal sheaf $\mathcal I_Z(\Theta)=\mathcal I_Z\otimes \mathcal O_A(\Theta)$ is a GV-sheaf. \label{item:GV-sheaf}
\end{enumerate}
Then $C$ is smooth and there is an isomorphism $(A,\Theta)\cong (J(C),\Theta_C)$ which identifies $C$, $Y$ and $Z$ with translates of $W_{1}(C)$, $W_{d-1}(C)$ and $W_d(C)$, respectively.
\end{theorem}

The sum of geometrically non-degenerate subvarieties $C,Y\subsetneq A$ of dimension $1$ and $d-1$ respectively yields a geometrically non-degenerate subvariety of dimension $d$, see Lemma \ref{lem:debarre} below. 
Therefore, any abelian variety contains lots of geometrically non-degenerate subvarieties $Z$ satisfying (\ref{item:Y=Y'+C}) in Theorem \ref{thm:GV}.

The point is property (\ref{item:GV-sheaf}) in Theorem \ref{thm:GV}. 
If $d=g-1$, where $g=\dim(A)$, this is known to be equivalent to $Z$ being a translate of $\Theta$, so we recover Theorem \ref{thm:Theta} from Theorem \ref{thm:GV}. 
If $1\leq d\leq g-2$, condition (\ref{item:GV-sheaf}) is more mysterious.
It is known to hold for $W_{d}(C)$ inside the Jacobian $J(C)$, as well as for the Fano surface of lines inside the intermediate Jacobian of a smooth cubic threefold.
Pareschi--Popa  conjectured (Conjecture \ref{conj:GV} below) that up to isomorphisms these are the only examples; they proved it for subvarieties of dimension one or codimension two.

\subsection{The DPC Problem for theta divisors} \label{subsec:intro:DPC}
A variety $X$ is DPC (dominated by a product of curves), if there are curves $C_1,\dots ,C_n$ together with a dominant rational map
\[
C_1\times\dots \times C_n \dashrightarrow X . 
\footnote{A priori $n\geq \dim(X)$, but by \cite[Lem.\ 6.1]{schoen}, we may actually assume $n=\dim(X)$.}
\]
For instance, unirational varieties, abelian varieties as well as Fermat hypersurfaces $\left\{x_0^d+\dots + x_N^d=0\right\} \subseteq \CP^N$ of degree $d\geq 1$ are DPC, see \cite{schoen}.
Serre \cite{grothendieck-serre} constructed the first example of a variety which is not DPC. 
Deligne \cite[Sec.\ 7]{deligne} and later Schoen \cite{schoen} used a Hodge theoretic obstruction to produce many more examples.

On the one hand, the theta divisor of the Jacobian of a smooth curve is DPC by Riemann's Theorem.
On the other hand, Schoen found \cite[p.\ 544]{schoen} that his Hodge theoretic obstruction does not even prevent smooth theta divisors from being DPC. 
This led Schoen \cite[Sec.\ 7.4]{schoen} to pose the problem of finding theta divisors which are not DPC, if such exist. 
The following solves that problem completely, which was our initial motivation for this paper.

\begin{corollary} \label{cor:ThetaDPC}
Let $(A,\Theta)$ be an indecomposable ppav.
The theta divisor $\Theta$ is DPC if and only if $(A,\Theta)$ is isomorphic to the Jacobian of a smooth curve. 
\end{corollary}

We prove in fact a strengthened version (Corollary \ref{cor:CxY}) of Corollary \ref{cor:ThetaDPC}, in which the DPC condition is replaced by the existence of a dominant rational map $Z_1\times Z_2\dashrightarrow \Theta$, where $Z_1$ and $Z_2$ are arbitrary varieties of dimension $1$ and $g-2$, respectively. 
The latter is easily seen to be equivalent to $\Theta$ having a curve summand and so Theorem \ref{thm:Theta} applies.

We discuss further applications of Theorem \ref{thm:Theta} in Sections \ref{subsec:DPCTheta} and \ref{subsec:martens}.
Firstly, using work of Clemens--Griffiths \cite{clemens-griffiths}, we prove that the Fano surface of lines on a smooth cubic threefold is not DPC (Corollary \ref{cor:Fano}).
Secondly, for a smooth genus $g$ curve $C$, we determine in Corollary \ref{cor:C^(k)} all possible ways in which the symmetric product $C^{(k)}$ with $k\leq g-1$ can be dominated by a product of curves.
Our result can be seen as a generalization of a theorem of Martens' \cite{martens,ran3}. 
 
\subsection{Method of proofs}
Although Theorem \ref{thm:Theta} is a special case of Theorem \ref{thm:GV}, it appears to be more natural to prove Theorem \ref{thm:Theta} first.
Here we use techniques that originated in work of Ran and Welters \cite{ran2,ran,welters}; they are mostly of cohomological and geometric nature. 
One essential ingredient is Ein--Lazarsfeld's result \cite{ein-laz} on the singularities of theta divisors, which allows us to make Welters' method \cite{welters} unconditional.
Eventually, Theorem \ref{thm:Theta} will be reduced to Matsusaka--Hoyt's criterion \cite{hoyt}, asserting that Jacobians of smooth curves are characterized among indecomposable $g$-dimensional ppavs $(A,\Theta)$ by the property that the cohomology class $\frac{1}{(g-1)!}[\Theta]^{g-1}$ can be represented by a curve.
Theorem \ref{thm:GV} follows then quickly from Theorem \ref{thm:Theta} and work of Debarre \cite{debarre} and Pareschi--Popa \cite{pareschi-popa}.

\subsection{Conventions}
We work over the field of complex numbers.
A variety is a separated integral scheme of finite type over $\C$;
if not mentioned otherwise, varieties are assumed to be proper over $\C$. 
A curve is an algebraic variety of dimension one.
In particular, varieties (and hence curves) are reduced and irreducible. 

If not mentioned otherwise, a point of a variety is always a closed point.
A general point of a variety or scheme is a closed point in some Zariski open and dense set.

For a codimension one subscheme $Z$ of a variety $X$, we denote by $\Wdiv_X(Z)$ the corresponding effective Weil divisor on $X$; if $Z$ is not pure-dimensional, all components of codimension $\geq 2$ are ignored in this definition.
Linear equivalence between divisors is denoted by $\sim$.

For subschemes $Z$ and $Z'$ of an abelian variety $A$, we denote by $Z+Z'$ (resp.\ $Z-Z'$) the image of the addition (resp.\ difference) morphism $Z\times Z'\longrightarrow A$, equipped with the natural image scheme structure. 
Note that for subvarieties $Z$ and $Z'$ of $A$, the image $Z\pm Z'$ is reduced and irreducible, hence a subvariety of $A$.
If $Z'$ is a point $a\in A$, $Z\pm Z'$ is also denoted by $Z_{\pm a}$.

If $Z\subseteq A$ is a subvariety of an abelian variety, the (Zariski) tangent space $T_{Z,z}$ at a point $z\in Z$ is identified via translation with a subspace of $T_{A,0}$.

\section{Non-degenerate subvarieties} \label{sec:absummand}

Following Ran \cite[p.\ 464]{ran}, a $d$-dimensional subvariety $Z$ of a $g$-dimensional abelian variety is called non-degenerate if the image of the Gauß map $G_Z:Z\dashrightarrow \Gr(d,g)$ is via the Plücker embedding not contained in any hyperplane.
This condition is stronger than the previously mentioned notion of geometrically non-degenerate subvarieties. 
We will need the following consequence of Lemma II.1 in \cite{ran}.

\begin{lemma} \label{lem:ran:non-deg}
Let $Z\subseteq A$ be a codimension $k$ subvariety of an abelian variety whose cohomology class is a multiple of $\frac{1}{k!}[\Theta]^{k}$.
Then $Z$ is non-degenerate, hence geometrically non-degenerate.
\end{lemma}

Ran proved that a $d$-dimensional subvariety $Z\subseteq A$ is geometrically non-degenerate if and only if for each abelian subvariety $B\subseteq A$, the composition $Z\longrightarrow A\slash B$ has either $d$-dimensional image or it is surjective \cite[Lem.\  II.12]{ran}.
In \cite[p.\ 105]{debarre:book}, Debarre used Ran's characterization as definition and proved the following. 

\begin{lemma} \label{lem:debarre}
Let $Z_1,Z_2\subseteq A$ be subvarieties of respective dimensions $d_1$ and $d_2$ with $d_1+d_2\leq \dim(A)$.
\begin{enumerate}
	\item If $Z_1$ is geometrically non-degenerate, $\dim(Z_1+Z_2)=d_1+d_2$.
	\item If $Z_1$ and $Z_2$ are geometrically non-degenerate, $Z_1+Z_2\subseteq A$ is geometrically non-degenerate.  
\end{enumerate} 
\end{lemma}

\section{A consequence of Ein--Lazarsfeld's Theorem} \label{sec:welters}
The purpose of this section is to prove Lemmas \ref{lem1} and \ref{lem2} below. 
Under the additional assumption
\begin{align} \label{eq:Sing(Theta)}
\dim(\Sing(\Theta))\leq \dim(A)-4 ,
\end{align}
these were first proven by Ran \cite[Cor.\ 3.3]{ran2} and Welters \cite[Prop.\ 2]{welters}, respectively.
The general case is a consequence of the following result of Ein--Lazarsfeld \cite{ein-laz}.

\begin{theorem}[Ein--Lazarsfeld] \label{thm:ein-laz} 
Let $(A,\Theta)$ be a ppav.
If $\Theta$ is irreducible, it is normal and has only rational singularities.
\end{theorem}

Let $(A,\Theta)$ be an indecomposable ppav of dimension $\geq 2$.
By the Decomposition Theorem \cite[p.\ 75]{birkenhake-lange}, $\Theta$ is irreducible and we choose a desingularization $f:X\longrightarrow \Theta$.
The composition of $f$ with the inclusion $\Theta\subseteq A$ is denoted by $j:X\longrightarrow A$.

\begin{lemma} \label{lem1}
Pullback of line bundles induces an isomorphism
\[
j^\ast:\Pic^0(A)\stackrel{\sim}\longrightarrow \Pic^0(X)  .
\]
\end{lemma}

\begin{proof} 
By Theorem \ref{thm:ein-laz}, $f_\ast \mathcal O_X=\mathcal O_\Theta$ and $R^if_\ast\mathcal O_X=0$ for all $i>0$.
We therefore obtain 
\[
H^1(X,\mathcal O_X)\cong H^1(\Theta,\mathcal O_\Theta) \cong H^1(A,\mathcal O_A) ,
\]
where the first isomorphism follows from the Leray spectral sequence, and the second one from Kodaira vanishing and the short exact sequence
\begin{align} \label{eq:ses}
0\longrightarrow \mathcal O_A(-\Theta)\longrightarrow \mathcal O_A\longrightarrow \mathcal O_\Theta =j_\ast \mathcal O_X \longrightarrow 0 .
\end{align} 
Hence, $j^\ast:\Pic^0(A)\longrightarrow \Pic^0(X)$ is an isogeny. 

Tensoring (\ref{eq:ses}) by a nontrivial $P\in \Pic^0(A)$, we obtain
\[
H^0(X,j^\ast P)\cong H^0(A,P)=0 ,
\]
where we applied Kodaira vanishing to  
$\mathcal O_A(-\Theta)\otimes P$.
It follows that $j^\ast P$ is nontrivial. 
That is, $j^\ast$ is an injective isogeny and thus an isomorphism.
This proves Lemma \ref{lem1}.
\end{proof}

\begin{lemma} \label{lem2}
For any $a\neq0$ in $A$, $j:X\longrightarrow A$ induces an isomorphism 
\[
j^\ast : H^0(A,\mathcal O_A(\Theta_a))\stackrel{\sim}\longrightarrow H^0(X,j^\ast(\mathcal O_A(\Theta_a))) .
\]
\end{lemma}
\begin{proof}
Following Welters \cite[Prop.\ 2]{welters}, the assertion follows from (\ref{eq:ses}) by tensoring with $\mathcal O_A(\Theta_a)$, since $\mathcal O_A(\Theta_a-\Theta)$ has no nonzero cohomology for $a\neq 0$.
\end{proof}

\section{Proof of Theorem 1} \label{sec:thm1}

Let $(A,\Theta)$ be a $g$-dimensional indecomposable ppav, and suppose that there is a curve $C\subseteq A$ and a $(g-2)$-dimensional subvariety $Y\subseteq A$ such that
\[
\Theta=C+Y .
\] 
After translation, we may assume $\Theta=-\Theta$. 
We pick a point $c_0\in C$ and replace $C$ and $Y$ by $C_{-c_0}$ and $Y_{c_0}$. 
Hence, $0\in C$ and so $Y=0+Y$ is contained in $\Theta$. 

Since $(A,\Theta)$ is indecomposable, $\Theta$ is irreducible, hence normal by Theorem \ref{thm:ein-laz}. 
The idea of the proof of Theorem \ref{thm:Theta} is to consider the intersection $\Theta\cap \Theta_{c}$ for nonzero $c\in C$.
Since $\Theta$ induces a principal polarization, $\Theta\cap \Theta_{c}$ is a proper subscheme of $\Theta$ for all $c\neq 0$.
For our purposes it is more convenient to consider the corresponding Weil divisor on $\Theta$, denoted by 
\[
\Wdiv_\Theta(\Theta\cap \Theta_{c}) .  
\]
Clearly, this divisor is just the pullback of the Cartier divisor $\Theta_c$ from $A$ to $\Theta$. 

Since $\Theta=-\Theta$, the map $x\mapsto c-x$ defines an involution of $\Theta\cap \Theta_{c}$.
Since $\Theta=C+Y$ and $0\in C$, it follows that $\Wdiv_\Theta(\Theta\cap \Theta_{c})$ contains the effective Weil divisors $Y_c$ and $-Y$.
For general $c$, these divisors are distinct and so we find
\begin{align} \label{eq:ThetacapTheta_c}
\Wdiv_\Theta(\Theta\cap \Theta_{c})= Y_c+ Z(c)  
\end{align}
for all $c\neq 0$, where $Z(c)$ is an effective Weil divisor on $\Theta$ which contains $-Y$:
\begin{align} \label{eq:-YsubsetZ(c)}
(-Y)\subseteq Z(c) .
\end{align} 
In the following proposition, we prove that actually $Z(c)=-Y$. 
As a byproduct of the proof, we are able to compute the cohomology class of $C$ in terms of the degree of the addition morphism
\[
F:C\times Y\longrightarrow \Theta .
\]  
Our proof uses Welters' method \cite{welters}.

\begin{proposition} \label{prop:[C],[Y]}
Let $(A,\Theta)$ be a $g$-dimensional indecomposable ppav with $\Theta=C+Y$, $\Theta=-\Theta$ and $0\in C$ as above. 
For any nonzero $c\in C$, 
\begin{align} \label{eq:Theta^2}
\Wdiv_\Theta(\Theta\cap \Theta_c) =Y_c + (-Y) .
\end{align}
Moreover, the cohomology class of $C$ is given by
\begin{align} \label{eq:[C]}
[C]=\frac{\deg(F)}{(g-1)^2\cdot (g-2)!}\cdot [\Theta]^{g-1} .
\end{align} 
\end{proposition}
 
\begin{proof} 
We fix a resolution of singularities $f:X\longrightarrow \Theta$ and denote the composition of $f$ with the inclusion $\Theta\subseteq A$ by $j:X\longrightarrow A$.
Moreover, for each $a\in A$, we fix some divisor $\widetilde\Theta_a$ on $X$ which lies in the linear series $|j^\ast(\Theta_a)|$.
For $a\neq 0$, $|j^\ast (\Theta_a)|$ is zero-dimensional by Lemma \ref{lem2}.
It follows that $\widetilde\Theta_a$ is unique if $a\neq 0$; it is explicitly given by
\begin{align} \label{eq:thetatilde}
\widetilde \Theta_a=\Wdiv_X(f^{-1}(\Theta_a\cap \Theta)) .
\end{align}  
Since $\Theta$ is normal, the general point of each component of $\Theta_a\cap \Theta$ lies in the smooth locus of $\Theta$.
The above description therefore proves
\begin{align} \label{eq:f_astTheta}
f_\ast \widetilde \Theta_a=\Wdiv_\Theta(\Theta_a\cap \Theta) ,
\end{align}
for all $a\neq 0$ in $A$.

Next, we would like to find a divisor $\widetilde Y_c$ on $X$ whose pushforward to $\Theta$ is $Y_c$. 
Since $Y_c$ is in general not Cartier on $\Theta$, we cannot simply take the pullback.
Instead, we consider the Weil divisor which corresponds to the scheme theoretic preimage of $Y_c$,
\begin{align} \label{def:Yctilde}
\widetilde Y_c:=\Wdiv_X(f^{-1}(Y_c)) .
\end{align} 
Since $\Theta$ is normal, $Y_c$ is not contained in the singular locus of $\Theta$.
It follows that $f^{-1}(Y_c)$ has a unique component which maps birationally onto $Y_c$ and the remaining components are in the kernel of $f_\ast$.
Hence,
\begin{align} \label{eq:f_astY}
f_\ast \widetilde Y_c=Y_c .
\end{align}
For all $c\neq 0$ in $C$, we define 
\begin{align} \label{def:Ztilde}
\widetilde Z(c):=\widetilde \Theta_c-\widetilde Y_c .
\end{align}
It follows from (\ref{eq:ThetacapTheta_c}), (\ref{eq:thetatilde}) and (\ref{def:Yctilde}) that $\widetilde Z(c)$ is effective. 
Moreover, by (\ref{eq:ThetacapTheta_c}), (\ref{eq:f_astTheta}) and (\ref{eq:f_astY}),
\begin{align} \label{eq:f_astZ}
f_\ast \widetilde Z(c)=\Wdiv_\Theta(\Theta\cap \Theta_c)- Y_c=Z(c) .
\end{align}

Consider the morphism $\varphi:X\times C \longrightarrow A$ with $\varphi(x,c):=f(x)-c$. 
The scheme theoretic preimage $\mathcal Y:=\varphi^{-1}(Y)$ has closed points 
$\left\{(x,c)\in X\times C\mid f(x)\in Y_c\right\} 
$ and the fibers of the second projection $\pr_2:\mathcal Y\longrightarrow C$ are given by $\pr_2^{-1}(c)\cong f^{-1}(Y_c)$.
By generic flatness applied to $\pr_2$, there is a Zariski dense and open subset $U\subseteq C$ such that the fibers $f^{-1}(Y_c)$ form a flat family for $c\in U$.
By the definition of $\widetilde Y_c$ in (\ref{def:Yctilde}), $\widetilde Y_c-\widetilde Y_{c'}$ is numerically trivial on $X$ for all $c,c'\in U$.
Lemma \ref{lem1} yields therefore for all $c,c'\in U$ a linear equivalence
\begin{align} \label{eq:z(c,c')}
\widetilde Y_{c}-\widetilde Y_{c'}\sim j^\ast(\Theta_{z(c,c')}-\Theta)\sim \widetilde \Theta_{z(c,c')}- \widetilde \Theta ,
\end{align}
where $z:U\times U\longrightarrow A$ is the morphism induced by the universal property of 
\[
\Pic^0(X)\cong \Pic^0(A) .
\] 

The proof of Proposition \ref{prop:[C],[Y]} proceeds now in several steps. 

\textbf{Step 1.}
Let $c'\in U$ and consider the function $x_{c'}(c):=z(c,c')+c'$.
For all $c\in U$ with $x_{c'}(c)\neq 0$, we have
\begin{align} \label{eq:step1}
\Wdiv_\Theta(\Theta_{x_{c'}(c)} \cap \Theta)= Y_{c} +  Z(c') .
\end{align}
Moreover, if $c'\in U$ is general, then $x_{c'}(c)$ is nonconstant in $c\in U$.

\begin{proof} 
Using the theorem of the square  \cite[p.\ 33]{birkenhake-lange} on $A$ and pulling back this linear equivalence to $X$ shows  
$
\widetilde \Theta_{x_{c'}(c)} \sim \widetilde \Theta_{z(c,c')}-\widetilde \Theta +\widetilde \Theta_{c'}
$.
By (\ref{eq:z(c,c')}) and the definition of $\widetilde Z(c')$ in (\ref{def:Ztilde}), we therefore obtain:  
\begin{align*}
\widetilde \Theta_{x_{c'}(c)} &\sim \widetilde \Theta_{z(c,c')}-\widetilde \Theta+\widetilde \Theta_{c'} \\
															&\sim \widetilde Y_{c}-\widetilde Y_{c'} + \widetilde \Theta_{c'} \\
															&\sim \widetilde Y_{c}+\widetilde Z(c') .
\end{align*}
That is, $\widetilde Y_{c}+\widetilde Z(c')$ is an effective divisor linearly equivalent to $\widetilde \Theta_{x_{c'}(c)}$.  
By Lemma \ref{lem2}, the linear series $|\widetilde \Theta_{x_{c'}(c)}|$ is zero-dimensional for all $x_{c'}(c)\neq 0$, and so we actually obtain an equality of Weil divisors:
\[
\widetilde \Theta_{x_{c'}(c)} = \widetilde Y_{c}+\widetilde Z(c') .
\]
Applying $f_\ast$ to this equality, (\ref{eq:step1}) follows from (\ref{eq:f_astTheta}), (\ref{eq:f_astY}) and (\ref{eq:f_astZ}).

Using again the theorem of the square on $A$ and pulling back the corresponding linear equivalence to $X$, we obtain
\[
\widetilde \Theta_{z(c,c')}- \widetilde \Theta \sim  \widetilde \Theta- \widetilde \Theta_{-z(c,c')} .
\]
It therefore follows from (\ref{eq:z(c,c')}) that $\widetilde \Theta_{-z(c,c')}\sim \widetilde \Theta_{z(c',c)}$.
By Lemma \ref{lem1}, $-z(c,c')=z(c',c)$.  

For a contradiction, suppose that $x_{c'}(c)=z(c,c')+c'$ is constant in $c$ for general (hence for all) $c'\in U$. 
It follows that $z(c,c')$ is constant in the first variable.
Since $z(c,c')=-z(c',c)$, it is also constant in the second variable. 
Therefore, for general $c'$, $x_{c'}(c)=z(c,c')+c'$ is nonzero (and constant in $c$).
This contradicts (\ref{eq:step1}), because its right hand side is nonconstant in $c$ as $C+Y=\Theta$.  
This concludes step 1. 
\end{proof} 

Let us now fix a general point $c'\in U$.
By step 1, the closure of $c\mapsto x_{c'}(c)$ is a proper irreducible curve $D\subseteq A$.

We say that a subvariety $Z$ of $A$ is translation invariant under $D$ if 
\[
Z_{x}=Z_{x'}
\] 
for all $x,x'\in D$.  
Equivalently, $Z$ is translation invariant under $D$ if and only if the corresponding cohomology classes on $A$ satisfy $[Z]\ast [D]=0$, where $\ast$ denotes the Pontryagin product.  
That description shows that the notion of translation invariance depends only on the cohomology classes of $Z$ and $D$.  
In particular, $Z$ is translation invariant under $D$ if and only if the same holds for $-Z$ or $-D$.
If $Z$ is not translation invariant under $D$, we also say that it moves when translated by $D$. 

For each $c\neq 0$, we decompose the Weil divisor $Z(c)$ on $\Theta$ into a sum of effective divisors
\begin{align} \label{def:Z_mov}
Z(c)=Z_{\mov}(c)+Z_{\inv}(c) ,
\end{align}
where $Z_{\inv}(c)$ contains all the components of $Z(c)$ that are translation invariant under $D$ and the components of $Z_{\mov}(c)$ move when translated by $D$.

We claim that the effective divisor $-Y$ is contained in $Z_{\mov}(c)$:
\begin{align} \label{eq:-YsubsetZ_mov(c)}
(-Y)\leq Z_{\mov}(c) .
\end{align}
Indeed, by (\ref{eq:-YsubsetZ(c)}), it suffices to prove that $-Y$ moves when translated by $D$. 
This follows as for $x_1,x_2\in D$ with $Y_{x_1}=Y_{x_2}$, 
\[
\Theta_{x_1}=C+Y_{x_1}=C+Y_{x_2}=\Theta_{x_2} ,
\]
and so $x_1=x_2$. 
 
\textbf{Step 2.} 
We have $x_{c'}(c)=c$ and hence $D=C$.
Moreover, for each $c\neq 0$ in $U$, 
\begin{align} \label{eq:step2}
\Wdiv_\Theta(\Theta\cap \Theta_{c})=Y_c+(-Y)+Z_{\inv}(c') .
\end{align}

\begin{proof}
Let $Z'$ be a prime divisor in $Z_{\mov}(c')$.
It follows from step 1 that $Z'_{-x}\subseteq \Theta$ for general $x\in D$, hence for all $x\in D$. 
Multiplication with $-1$ shows $(-Z')_x\subseteq -\Theta=\Theta$ for all $x\in D$. 
Since $-Z'\subseteq -\Theta=\Theta$, this equality implies
\[
(-Z')_{x}\subseteq \Theta_{x}\cap\Theta
\]
for all $x\in D$.
Therefore, for each $c\in U$ with $x_{c'}(c)\neq 0$, the prime divisor $(-Z')_{x_{c'}(c)}$ is contained in  $\Wdiv_\Theta(\Theta_{x_{c'}(c)}\cap\Theta)$.
Hence, by (\ref{eq:step1}) from step 1,
\begin{align} \label{ineq:step2}
(-Z')_{x_{c'}(c)}\leq Y_{c} +  Z(c') ,
\end{align}
for all $c\in U$ with $x_{c'}(c)\neq 0$.

Let us consider (\ref{ineq:step2}), where we move the point $c$ in $C$ and keep $c'$ fixed and general.
By step 1, the point $x_{c'}(c)$ moves. 
Since $Z'$ is a component of $Z_{\mov}(c')$, the translate $(-Z')_{x_{c'}(c)}$ must also move.
The translate $Y_c$ moves because $Y+C=\Theta$. 
Clearly, $Z(c')$ does not move as we keep $c'$ fixed.
By (\ref{ineq:step2}),
\begin{align} \label{eq:step2:Z'}
(-Z')_{x_{c'}(c)}=Y_{c} .
\end{align} 

By (\ref{eq:-YsubsetZ_mov(c)}), equality (\ref{eq:step2:Z'}) holds for $Z'=-Y$, which proves $Y_{x_{c'}(c)}=Y_{c} $. 
This implies
\[
\Theta_{x_{c'}(c)}=Y_{x_{c'}(c)}+C=Y_{c}+C=\Theta_c .
\]
Hence, 
\[
x_{c'}(c)=c ,
\]
which proves $D=C$.

It remains to prove (\ref{eq:step2}).
Since $x_{c'}(c)=c$, (\ref{eq:-YsubsetZ_mov(c)}) and (\ref{eq:step2:Z'}) show that $-Y$ is actually the only prime divisor in $Z_{\mov}(c')$. 
Hence,
\[
Z_{\mov}(c')=\lambda \cdot (-Y)
\]
for some positive integer $\lambda$.
Using $x_{c'}(c)=c$ and (\ref{def:Z_mov}) in the conclusion (\ref{eq:step1}) from step 1, we therefore obtain
\[
\Wdiv_\Theta(\Theta\cap \Theta_{c})=Y_c+\lambda\cdot(-Y)+Z_{\inv}(c') .
\]

For (\ref{eq:step2}), it now remains to prove $\lambda=1$.
That is, it suffices to prove that for general points $y\in Y$ and $c\in C$, the intersection $\Theta\cap \Theta_{c}$ is transverse at the point $-y$.
Recall that $\Theta$ is normal and so it is smooth at $-y$ for $y\in Y$ general.
It thus suffices to see that the tangent space $T_{\Theta,-y}$ meets $T_{\Theta_{c},-y}=T_{\Theta,-y-c}$ properly.  
Since $T_{\Theta,-y}$ and $T_{\Theta,-y-c}$ have codimension one in $T_{A,0}$, it actually suffices to prove
\[
T_{\Theta,-y}\neq T_{\Theta,-y-c}   
\]
for general $c\in C$ and $y\in Y$. 
In order to see this, it suffices to note that $\Theta$ is irreducible and so the Gauß map 
\[
G_\Theta:\Theta \dashrightarrow \CP^{g-1}
\]
is generically finite \cite[Prop.\ 4.4.2]{birkenhake-lange}. 
Indeed, $T_{\Theta,-y}=T_{\Theta,-y-c}$ for general $c$ and $y$ implies that through a general point of $\Theta$ (which is of the form $-y-c$) there is a curve which is contracted by $G_{\Theta}$.  
This concludes step 2.
\end{proof}

\textbf{Step 3.}  
We have the following identity in $H^{2g-2}(A,\Z)$:
\begin{align} \label{eq:step3}
[\Theta]^2\ast [C] = 2\cdot \deg(F)\cdot [\Theta] ,
\end{align}
where we recall that $F:C\times Y\longrightarrow \Theta$ denotes the addition morphism.

\begin{proof}
It follows from the conclusion (\ref{eq:step2}) in step 2 that $Z_{\inv}(c')$ is actually independent of the general point $c'\in U$.
We therefore write $Z_{\inv}=Z_{\inv}(c')$.

Suppose that there is a prime divisor $Z'\leq Z_{\inv}$ on $\Theta$.
Let us think of $Z'$ as a codimension two cycle on $A$.
By definition, $Z'$ is translation invariant under $D$, hence under $C$ by step 2.
Therefore, $[Z']\ast [C]=0$ in $H^{2g-2}(A,\Z)$.
This holds for each prime divisor $Z'$ in $Z_{\inv}$, hence
\[
[Z_{\inv}]\ast [C]=0 .
\]

For $c\neq 0$, we may consider $\Theta\cap \Theta_c$ as a pure-dimensional codimension two subscheme of $A$.
As such it gives rise to an effective codimension two cycle on $A$, which is nothing but the pushforward of the cycle $\Wdiv_\Theta(\Theta\cap \Theta_c)$ from $\Theta$ to $A$.
Mapping this cycle further to cohomology, we obtain $[\Theta]^2$ in $H^{2g-4}(A,\Z)$.
Conclusion (\ref{eq:step2}) in step 2 therefore implies
\begin{align*}
[\Theta]^2\ast [C]&=2\cdot [Y]\ast [C]+[Z_{\inv}]\ast [C] \\
									&=2\cdot [Y]\ast [C] \\
									&=2\cdot \deg(F)\cdot [\Theta] ,
\end{align*}
where we used $[Y]=[Y_c]=[-Y]$ and $[Z_{\inv}]\ast [C]=0$.
\end{proof}

\textbf{Step 4.} Assertion (\ref{eq:[C]}) of Proposition \ref{prop:[C],[Y]} holds.  
 
\begin{proof}
We apply the cohomological Fourier--Mukai functor to the conclusion (\ref{eq:step3}) of step 3.
Using Lemma 9.23 and Lemma 9.27 in \cite{huy}, this yields:
\begin{align} \label{eq:degF}
\frac{2}{(g-2)!}\cdot [\Theta]^{g-2}\cup \PD [C] =\frac{2\cdot \deg(F)}{(g-1)!}\cdot [\Theta]^{g-1} ,
\end{align}
where $\PD$ denotes the Poincar\'e duality operator. 
Here we used 
\[
\PD\left(\frac{1}{k!}\cdot [\Theta]^k\right)=\frac{1}{(g-k)!}\cdot [\Theta]^{g-k} 
\]
for all $0\leq k\leq g$.

By the Hard Lefschetz Theorem, (\ref{eq:degF}) implies 
\[
[C]=\frac{\deg(F)}{(g-1)^2\cdot (g-2)!}\cdot [\Theta]^{g-1} ,
\]
which is precisely assertion (\ref{eq:[C]}) of Proposition \ref{prop:[C],[Y]}.
\end{proof}

By Lemma \ref{lem:ran:non-deg}, assertion (\ref{eq:[C]}) of Proposition \ref{prop:[C],[Y]} implies that $C$ is geometrically non-degenerate.
It follows from Lemma \ref{lem:debarre} that no proper subvariety of $A$ is translation invariant under $C$, hence under $D$ by the second conclusion of step 2.
This implies $Z_{\inv}(c')=0$ by its definition in (\ref{def:Z_mov}). 
Assertion (\ref{eq:Theta^2}) of Proposition \ref{prop:[C],[Y]} follows therefore from assertion (\ref{eq:step2}) in step 2. 
This finishes the proof of Proposition \ref{prop:[C],[Y]}.
\end{proof}

The next step in the proof of Theorem \ref{thm:Theta} is the following 

\begin{proposition} \label{prop:a=1} 
In the same notation as above, $C$ is smooth, $\deg(F)=g-1$ and $[C]=\frac{1}{(g-1)!}\cdot [\Theta]^{g-1} $.
\end{proposition}

\begin{proof} 
Let us first show that $C$ is smooth.
Indeed, (\ref{eq:Theta^2}) implies by Lemma \ref{lem:ran:non-deg} that $Y$ is non-degenerate. 
Via the Plücker embedding, its Gauß image is therefore not contained in any hyperplane.
If $c_0\in C$ is a singular point, the sum of Zariski tangent spaces $T_{C,c_0}+T_{Y,y}$ has thus for general $y\in Y$ dimension $g$.
It follows that $c_0+Y$ is contained in the singular locus of $\Theta$, which contradicts its normality (Theorem \ref{thm:ein-laz}).
Therefore $C$ is smooth.

In order to prove Proposition \ref{prop:a=1}, it suffices by (\ref{eq:[C]}) to show $\deg(F)=g-1$.
This will be achieved by computing the degree of $i^\ast \Theta$, where $i:C\longrightarrow A$ denotes the inclusion, in two ways. 
On the one hand, (\ref{eq:[C]}) implies
\begin{align} \label{eq:deg(L|_C)}
\deg\left(i^\ast \Theta\right)=[C]\cup [\Theta]=\frac{\deg(F)}{(g-1)^2\cdot (g-2)!}[\Theta]^{g}=\frac{g\cdot \deg(F)}{{g-1}} .
\end{align}

On the other hand, we may consider the addition morphism
$
m:C\times C\times Y \longrightarrow A 
$. 
For $y\in Y$, the restriction of $m$ to $C\times C \times y$ will be denoted by 
\[
m_y:C\times C\longrightarrow A .
\]

Since the degree is constant in flat families, we obtain
\begin{align} \label{eq2:deg(L|C)}
\deg(i^\ast \Theta)=\deg(i^\ast(\Theta_{-c-y}))=\deg \left(\left(m^\ast_{y} \Theta\right)|_{C\times c}\right)
\end{align}
for all $c\in C$ and $y\in Y$.

Let us now fix a general point $y\in Y$. 
Then the image of $m_y$ is not contained in $\Theta$ because $C+C+Y=A$.
Therefore, we can pull back the Weil divisor $\Theta$ as
\[
m_y^\ast (\Theta)=\Wdiv_{C\times C}(m_y^{-1}(\Theta)) ,
\] 
where $m_y^{-1}(\Theta)$ denotes the scheme-theoretic preimage, whose closed points are given by
\[
\left\{(c_1,c_2)\in C\times C  \mid c_1+c_2+y\in \Theta\right\} .
\] 
Hence, $m_y^\ast (\Theta)$ contains the prime divisors $C\times 0$ and $0\times C$.
We aim to calculate the right hand side of (\ref{eq2:deg(L|C)}) and proceed again in several steps.
 
\textbf{Step 1.} The multiplicity of $C\times0$ and $0\times C$ in $m_y^\ast(\Theta)$ is one. 

\begin{proof}
Let $\lambda$ be the multiplicity of $C\times 0$ in $m_y^{\ast}(\Theta)$. 
For $c\in C$ general, the point $(c,0)$ has then multiplicity $\lambda$ in the $0$-dimensional scheme
\[
m_y^{-1} (\Theta)\cap(c\times C) .
\]
Since $m_y$ maps $c\times C$ isomorphically to $C_{c+y}$, the above scheme is isomorphic to 
\[
\Theta\cap(C_{c+y}) ,
\]
and $c+y\in C_{c+y}$ has multiplicity $\lambda$ in that intersection.
If $\lambda\geq 2$, then
\[
T_{C,0}=T_{C_{c+y},c+y}\subseteq T_{\Theta,c+y} .
\] 
Since $c+y$ is a general point of $\Theta$, this inclusion contradicts the previously mentioned fact that the Gauß map $G_\Theta$ is generically finite and so the tangent space of $\Theta$ at a general point does not contain a fixed line.
This proves that $C\times 0$ has multiplicity one in $m_y^{\ast}(\Theta)$.
A similar argument shows that the same holds for $0\times C$, which concludes step 1. 
\end{proof}

By step 1,
\begin{align} \label{eq:Gamma}
m_y^{\ast}(\Theta)=\Wdiv_{C\times C}(m_y^{-1}(\Theta))=\left(C\times 0\right)+\left(0\times C\right) +\Gamma
\end{align}
for some effective $1$-cycle $\Gamma$ on $C\times C$ which contains neither $C\times 0$ nor $0\times C$.

\textbf{Step 2.}
Let $\Gamma'$ be a prime divisor in $\Gamma$.
Then for each $(c_1,c_2)\in \Gamma'$, 
\begin{align} \label{eq:prop2:step2}
-c_1-c_2-y \in Y .
\end{align}

\begin{proof}
Condition (\ref{eq:prop2:step2}) is Zariski closed and so it suffices to prove it for a general point $(c_1,c_2)\in \Gamma'$.
Such a point satisfies $c_1\neq0\neq c_2$ and $c_1+c_2+y\in \Theta\cap \Theta_{c_i}$ for $i=1,2$. 
We can therefore apply (\ref{eq:Theta^2}) in Proposition \ref{prop:[C],[Y]} and obtain
\[
c_1+c_2+y\in \supp(Y_{c_i}  + (-Y))  ,
\]
for $i=1,2$, where $\supp(-)$ denotes the support of the corresponding effective Weil divisor. 
It follows that $c_1+c_2+y$ lies in $Y_{c_1}\cap Y_{c_2}$ or in $(-Y)$.

We need to rule out $c_1+c_2+y\in Y_{c_1}\cap Y_{c_2}$. 
But if this is the case, then $c_1+y$ and $c_2+y$ are both contained in $Y$. 
Since $y\in Y$ is general, the intersection $\left(C+y\right)\cap Y$ is proper and so $(c_1,c_2)$ is contained in a finite set of points, which contradicts the assumption that it is a general point of $\Gamma'$.
This concludes step 2.
\end{proof}

\textbf{Step 3.} 
The $1$-cycle $\Gamma$ is reduced.

\begin{proof}
In order to see that $\Gamma$ is reduced, it suffices to prove that the intersections of $m_y^{-1}(\Theta)$ with $c \times C$ and $C\times c $ are both reduced, where $c \in C$ is general.
The other assertion being similar, we will only prove that $m_y^{-1}(\Theta) \cap(C\times c_2)$ is reduced, where $c_2\in C$ is general. 
Since $m_y$ maps $C\times c_2$ isomorphically to $C_{c_2+y}$, it suffices to prove that the intersection 
\begin{align} \label{eq:cap=transv}
C_{c_2+y}\cap \Theta
\end{align}
is transverse, where $c_2\in C$ and $y\in Y$ are both general. 

Let us consider a point $c_1\in C$ with $c_1+c_2+y\in\Theta$.
For $c_1=0$, transversality of (\ref{eq:cap=transv}) in $c_1+c_2+y$ was proven in step 1.
For $c_1\neq 0$, step 2 implies that $y_1:=-(c_1+c_2+y)$ is contained in $Y$.
In order to prove that the intersection (\ref{eq:cap=transv}) is transverse at $-y_1$, we need to see that
\begin{align} \label{eq:T_Cc_1}
T_{C,c_1}=T_{C_{c_2+y},-y_1}\subsetneq T_{\Theta,-y_1} . 
\end{align}
This follows from the fact that $c_2$ and $y$ are general as follows.

Recall the addition map $m:C\times C\times Y\longrightarrow A$ and 
consider the scheme theoretic preimage $m^{-1}(-Y)$ together with the projections
\[
\pr_{23}:m^{-1}(-Y)\longrightarrow C\times Y \ \ \text{and}\ \ \pr_3:m^{-1}(-Y)\longrightarrow Y .
\]
Let $\Gamma'$ be a prime divisor in $\Gamma$ with $(c_1,c_2)\in \Gamma'$.
It follows from step 2 that $\Gamma'\times y$ is contained in some component $Z$ of $m^{-1}(-Y)$.
The restriction of $\pr_{23}$ to $Z$ is surjective because $c_2$ and $y$ are general.
Hence, $\dim(Z)>\dim(Y)$ and so there is a curve in $Z$ passing through $(c_1,c_2,y)$ which is contracted via $m$ to $y_1$.
That is, there is some quasi-projective curve $T$ together with a nonconstant morphism $(\tilde c_1,\tilde c_2,\tilde y):T\longrightarrow C\times C\times Y$, with $\tilde c_1(t_0)=c_1$, $\tilde c_2(t_0)=c_2$ and $\tilde y(t_0)=y$ for some $t_0\in T$ such that
\[
\tilde c_1(t)+\tilde c_2(t)+\tilde y(t)=-y_1 ,
\]
for all $t\in T$.
Since $c_2\in C$ and $y\in Y$ are general, it follows that the addition morphism $F:C\times Y\longrightarrow \Theta$ is generically finite in a neighbourhood of $(c_2,y)$.
Hence, 
\[
\tilde c_1(t)=-y_1-\tilde c_2(t)-\tilde y(t)
\] 
is nonconstant in $t$.

For a contradiction, suppose $T_{C,c_1}\subset T_{\Theta,-y_1}$, where we recall $-y_1=c_1+c_2+y$. 
The image of $(\tilde c_2,\tilde y):T\longrightarrow C\times Y$ is a curve through the general point $(c_2,y)$.
It follows that $(\tilde c_2(t),\tilde y(t))$ is a general point of $C\times Y$ for general $t\in T$.
Replacing $(c_2,y)$ by $(\tilde c_2(t),\tilde y(t))$ in the above argument therefore shows
\[
T_{C,\tilde c_1(t)}\subset T_{\Theta,-y_1}
\]
for general (hence all) $t\in T$, since $-y_1=\tilde c_1(t)+\tilde c_2(t) +\tilde y(t)$.
As $\tilde c_1(t)$ is nonconstant in $t$, $T_{C,c}$ is contained in the plane $T_{\Theta,-y_1}$ for general $c\in C$. 
Hence, $C$ is geometrically degenerate, which by Lemma \ref{lem:ran:non-deg} contradicts (\ref{eq:[C]}) in Proposition \ref{prop:[C],[Y]}.
This contradiction establishes (\ref{eq:T_Cc_1}), which finishes the proof of step 3. 
\end{proof}

\textbf{Step 4.} For $c_2\in C$ general, $\deg(\Gamma |_{C\times c_2})= \deg(F)$. 

\begin{proof}
Let $c_2\in C$ be general.
By step 3, $\Gamma$ is reduced and so its restriction to $C\times c_2$ is a reduced $0$-cycle.
Since $c_2$ and $y$ are general, $-c_2-y$ is a general point of $\Theta$.
Therefore, $F^{-1}(-c_2-y)$ is also reduced.
It thus suffices to construct a bijection between the closed points of the zero-dimensional reduced schemes $\supp(\Gamma)\cap \left(C\times c_2\right)$ and $F^{-1}(-c_2-y)$. 
This bijection is given by 
\[
\phi:\supp(\Gamma)\cap\left( C\times c_2\right)\longrightarrow F^{-1}(-c_2-y) ,
\]
where $\phi((c_1,c_2))=(c_1,-c_1-c_2-y)$.
The point is here that $\phi$ is well-defined by step 2; its inverse is given by
\[
\phi^{-1}((c_1,y_1))=(c_1,-c_1-y_1-y) .
\] 
This establishes the assertion in step 4.
\end{proof}

By step 4, $\deg(\Gamma |_{C\times c_2})= \deg(F)$ for a general point $c_2\in C$.
Using (\ref{eq2:deg(L|C)}) and (\ref{eq:Gamma}), we obtain therefore
\[
\deg\left(i^\ast\Theta\right)=1+\deg(\Gamma|_{C\times c_2})=1+\deg(F) .
\]
Comparing this with (\ref{eq:deg(L|_C)}) yields 
\[
\frac{g\cdot \deg(F)}{g-1}=1+\deg(F),
\]
hence $\deg(F)=g-1$, as we want.
This finishes the proof of Proposition \ref{prop:a=1}.
\end{proof}

\begin{proof}[Proof of Theorem \ref{thm:Theta}]
Let $(A,\Theta)$ be an indecomposable ppav with $\Theta=C+Y$.
As explained in the beginning of Section \ref{sec:thm1}, we may assume $\Theta=-\Theta$ and $0\in C$.
By Proposition \ref{prop:a=1} and Matsusaka--Hoyt's criterion \cite[p.\ 416]{hoyt}, $C$ is smooth and there is an isomorphism $\psi:(A,\Theta)\congpf (J(C),\Theta_C)$ which maps $C$ to a translate of $W_1(C)$. 
Since $0\in C$, it follows that $\psi(C)= W_1(C)-x_2$ for some $x_2\in W_1(C)$. 

For $x_1\in W_1(C)$ with $x_1\neq x_2$, Weil \cite{weil} proved 
\begin{align} \label{eq:reducible}
\Wdiv_{W_{g-1}(C)}(W_{g-1}(C)\cap W_{g-1}(C)_{x_1-x_2})=W_{g-2}(C)_{x_1}+ (-W_{g-2}(C))_{-{\kappa}-x_2} ,
\end{align}
where ${\kappa}\in J(C)$ is such that $-W_{g-1}(C)=W_{g-1}(C)_{\kappa}$.
We move $x_1$ in $W_1(C)$ and compare (\ref{eq:Theta^2}) with (\ref{eq:reducible}) to conclude that $\psi(Y)$ is a translate of $W_{g-2}(C)$.
This finishes the proof of Theorem \ref{thm:Theta}.
\end{proof}

\begin{remark} \label{rem:shortcut}
Welters \cite[p.\ 440]{welters} showed that the conclusion of Proposition \ref{prop:[C],[Y]} implies the existence of a positive-dimensional family of trisecants of the Kummer variety of $(A,\Theta)$.
The latter characterizes Jacobians by results of Gunning's \cite{gunning} and Matsusaka--Hoyt's \cite{hoyt} and could hence be used to circumvent Proposition \ref{prop:a=1} in the proof of Theorem \ref{thm:Theta}. 
We presented Proposition \ref{prop:a=1} here because its proof is elementary and purely algebraic, whereas the use of trisecants involves analytic methods, see \cite{gunning,krichever}. 
It is hoped that this might be useful in other situations (e.g.\ in positive characteristics) as well. 
We also remark that Proposition \ref{prop:a=1} can be used to avoid the use of Gunning's results in Welters' work \cite{welters}.
\end{remark}

\begin{remark} \label{rem:little}
In \cite[p.\ 254]{little2}, Little conjectured Theorem \ref{thm:Theta} for $g=4$; a proof is claimed if $\Theta=C+S$ is a sum of a curve $C$ and a surface $S$, where no translate of $C$ or $S$ is symmetric (hence $C$ is non-hyperelliptic) and some additional non-degeneracy assumptions hold. 
However, some parts of the proof seem to be flawed and so further assumptions on $C$ and $S$ are necessary in \cite{little2}, see \cite{little3}.
\end{remark}

\section{GV-sheaves, theta duals and Pareschi--Popa's conjectures} \label{sec:pp}

The purpose of this section is to prove Theorem \ref{thm:GV} stated in the introduction and to explain two related conjectures of Pareschi and Popa. 
We need to recall some results of Pareschi--Popa's work \cite{pareschi-popa} first.

Let $(A,\Theta)$ be a ppav of dimension $g$.
By \cite[Thm.\ 2.1]{pareschi-popa}, a coherent sheaf $\mathcal F$ on $A$ is a GV-sheaf if and only if the complex
\begin{align} \label{eq:def:GV}
\R\hat{\mathcal S}(\R {\mathcal Hom}(\mathcal F,\mathcal O_A)) 
\end{align}
in the derived category of the dual abelian variety $\hat A$ has zero cohomology in all degrees $i\neq g$. 
Here, $\R\hat{\mathcal S}:\D^b(A)\longrightarrow \D^b(\hat A)$ denotes the Fourier--Mukai transform with respect to the Poincar\'e line bundle \cite[p.\ 201]{huy}.

For a geometrically non-degenerate subvariety $Z\subseteq A$, Pareschi and Popa consider the twisted ideal sheaf $\mathcal I_Z(\Theta)=\mathcal I_Z\otimes \mathcal O_A(\Theta)$.\footnote{In fact, Pareschi and Popa treat the more general case of an equidimensional closed reduced subscheme $Z\subseteq A$, but for our purposes the case of subvarieties will be sufficient.}
It follows from their own and Höring's work respectively that this is a GV-sheaf if $Z$ is a translate of $W_d(C)$ in the Jacobian of a smooth curve or of the Fano surface of lines in the intermediate Jacobian of a smooth cubic three-fold, see \cite[p.\ 210]{pareschi-popa}.
Both examples are known to have minimal cohomology class $\frac{1}{(g-d)!} [\Theta]^{g-d}$.
Pareschi--Popa's Theorem \cite[Thm.\ B]{pareschi-popa} says that this holds in general:

\begin{theorem}[Pareschi--Popa]  \label{thm:pareschi-popa}
Let $Z$ be a $d$-dimensional geometrically non-degenerate subvariety of a $g$-dimensional ppav $(A,\Theta)$.
If $\mathcal I_Z(\Theta)$ is a GV-sheaf, 
\[
[Z]=\frac{1}{(g-d)!} [\Theta]^{g-d} .
\]
\end{theorem}

Combining Theorem \ref{thm:pareschi-popa} with Debarre's ``minimal class conjecture'' in \cite{debarre}, Pareschi and Popa arrive at the following, see \cite[p.\ 210]{pareschi-popa}.

\begin{conjecture}\label{conj:GV}
Let $(A,\Theta)$ be an indecomposable ppav of dimension $g$ and let $Z$ be a geometrically non-degenerate $d$-dimensional subvariety with $1\leq d\leq g-2$.
If 
\begin{align} \label{eq:I_Y=GV}
\mathcal I_Z(\Theta) \ \text{is a GV-sheaf} ,
\end{align}
then either $(A,\Theta)$ is isomorphic to the Jacobian of a smooth curve $C$ and $Z$ is a translate of $W_d(C)$, or it is isomorphic to the intermediate Jacobian of a smooth cubic threefold and $Z$ is a translate of the Fano surface of lines.
\end{conjecture}

Pareschi and Popa \cite[Thm.\ C]{pareschi-popa} proved Conjecture \ref{conj:GV} for $d=1$ and $d=g-2$.
Theorem \ref{thm:GV} stated in the introduction proves it for subvarieties with curve summands and arbitrary dimension.
Before we can explain the proof of Theorem \ref{thm:GV}, we need to recall Pareschi--Popa's notion of theta duals \cite[p.\ 216]{pareschi-popa}.

\begin{definition}
Let $Z\subseteq A$ be a subvariety.
Its theta dual $\mathcal V(Z)\subseteq \hat A$ is the scheme-theoretic support of the $g$-th cohomology sheaf of the complex
\[
(-1_{\hat A})^\ast \R\hat{\mathcal S}(\R\mathcal Hom(\mathcal I_Z(\Theta),\mathcal O_A)) 
\]
in the derived category $\D^b(\hat A)$.
\end{definition}

From now on, we use $\Theta$ to identify $\hat A$ with $A$. 
The theta dual of $Z\subseteq A$ is then a subscheme $\T(Z)\subseteq A$.
For $W_d(C)$ inside a Jacobian of dimension $g\geq 2$, Pareschi and Popa proved \cite[Sect.\ 8.1]{pareschi-popa}
\begin{align} \label{eq:T(W_g-2)}
\T(W_d(C))=-W_{g-d-1}(C), 
\end{align} 
for $1\leq d\leq g-2$.
Apart from this example, it is in general difficult to compute $\T(Z)$.
However, the reduced scheme $\T(Z)^{\reduced}$ can be easily described as follows.

\begin{lemma} \label{lem:thetadual}
Let $Z\subseteq A$ be a subvariety.
The components of the reduced scheme $\T(Z)^{\reduced}$ are given by the maximal (with respect to inclusion) subvarieties $W\subseteq A$ such that $Z-W\subseteq \Theta$.
\end{lemma}
\begin{proof}
By \cite[p.\ 216]{pareschi-popa}, the set of closed points of $\T(Z)$ is $\left\{a\in A\mid Z\subseteq \Theta_a\right\}$.
This proves the lemma.
\end{proof}

We will use the following consequence of (\ref{eq:T(W_g-2)}) and Lemma \ref{lem:thetadual}.

\begin{lemma} \label{lem:W=W_d}
Let $C$ be a smooth curve of genus $g\geq 2$ and let $Z$ be a $(g-d-1)$-dimensional subvariety of $J(C)$ such that $W_d(C)+Z$ is a translate of the theta divisor $\Theta_C$.
Then, $Z$ is a translate of $W_{g-d-1}(C)$. 
\end{lemma}
\begin{proof}
By assumption, there is a point $a\in J(C)$ with 
$
W_d(C)+Z_a=\Theta_C 
$.
Hence, 
\[
(-Z)_{-a}\subseteq \T(W_d(C))
\] 
by Lemma \ref{lem:thetadual}.
By (\ref{eq:T(W_g-2)}), $(-Z)_{-a}\subseteq -W_{g-d-1}(C)$ and equality follows because of dimension reasons.
\end{proof}

For a geometrically non-degenerate subvariety $Z\subseteq A$ of dimension $d$,
\begin{align} \label{ineq:dim(T(Y))}
\dim(\T(Z))\leq g-d-1 
\end{align}
follows from Lemmas \ref{lem:debarre} and \ref{lem:thetadual}. 
Moreover, if equality is attained in (\ref{ineq:dim(T(Y))}), then
$
\Theta=Z-W
$ 
for some component $W$ of $\T(Z)^{\reduced}$, and so $\Theta$ has $Z$ as a $d$-dimensional summand. 

Pareschi and Popa proved the following \cite[Thm.\ 5.2(a)]{pareschi-popa}.

\begin{proposition} \label{prop:pp} 
Let $Z\subseteq A$ be a geometrically non-degenerate subvariety.
If $\mathcal I_Z(\Theta)$ is a GV-sheaf, equality holds in (\ref{ineq:dim(T(Y))}).
\end{proposition}

Motivated by Proposition \ref{prop:pp}, Pareschi and Popa conjectured \cite[p.\ 222]{pareschi-popa} that Conjecture \ref{conj:GV} holds if one replaces (\ref{eq:I_Y=GV}) by the weaker assumption 
\begin{align} \label{eq:dim(T(Y))}
\dim(\T(Z))=g-d-1 .
\end{align}
By the above discussion, their conjecture is equivalent to

\begin{conjecture}\label{conj:pp} 
Let $(A,\Theta)$ be an indecomposable ppav of dimension $g$ and let $Z$ be a geometrically non-degenerate subvariety of dimension $1\leq d\leq g-2$.
Suppose that
\begin{align} \label{eq:conjpp}
\Theta=Z+W 
\end{align}
for some subvariety $W\subseteq A$.
Then, either $(A,\Theta)$ is isomorphic to the Jacobian of a smooth curve $C$ and $Z$ is a translate of $W_d(C)$, or it is isomorphic to the intermediate Jacobian of a smooth cubic threefold and $Z$ is a translate of the Fano surface of lines.
\end{conjecture} 

Theorem \ref{thm:Theta} proves (a strengthening of) Conjecture \ref{conj:pp} for $d=1$ and $d=g-2$.
This provides the first known evidence for that conjecture.

\begin{remark}
Conjecture \ref{conj:GV} is implied by Conjecture \ref{conj:pp}, as well as by Debarre's ``minimal class conjecture'' in \cite{debarre}.
Similar implications among the latter two conjectures are not known.
\end{remark}


We end this section with the proof of Theorem \ref{thm:GV}.

\begin{proof}[Proof of Theorem \ref{thm:GV}]
Let $Z\subsetneq A$ be as in Theorem \ref{thm:GV}.
Since $\mathcal I_Z(\Theta)$ is a GV-sheaf, equality holds in (\ref{ineq:dim(T(Y))}) by Proposition \ref{prop:pp}. 
The reduced theta dual $\T(Z)^{\reduced}$ contains thus by Lemmas \ref{lem:debarre} and \ref{lem:thetadual} a $(g-d-1)$-dimensional component $W$ with $Z-W=\Theta$.
By assumption (\ref{item:Y=Y'+C}) in Theorem \ref{thm:GV}, we obtain
\begin{align*} 
\Theta= C+Y-W .
\end{align*} 
By Theorem \ref{thm:Theta}, $C$ is smooth and there is an isomorphism $\psi:(A,\Theta)\congpf (J(C),\Theta_C)$ which identifies $C$ and $Y-W$ with translates of $W_1(C)$ and $W_{g-2}(C)$, respectively. 
Hence,
\begin{align} \label{eq:psiZ-psiW}
\psi(Z)-\psi(W)=\psi(C)+\psi(Y)-\psi(W)= W_{g-1}(C)_a ,
\end{align}
for some $a\in J(C)$ and it remains to prove that $\psi(Y)$ is a translate of $W_{d-1}(C)$.

If $d=g-1$, then $\psi(W)$ is a point and $\psi(Y)$ is a translate of $W_{g-2}(C)$, as we want.
We may therefore assume $d\leq g-2$ in the following.
By Theorem \ref{thm:pareschi-popa}, the GV-condition on $\mathcal I_Z(\Theta)$ implies
\begin{align*} 
[Z]=\frac{1}{(g-d)!}\cdot [\Theta]^{g-d} .
\end{align*}
By Debarre's Theorem \cite{debarre}, $\psi(Z)$ is thus a translate of $W_d(C)$ or $-W_d(C)$.

\textbf{Case 1:} $\psi(Z)$ is a translate of $W_d(C)$.

By (\ref{eq:psiZ-psiW}), $W_d(C)-\psi(W)$ is here a translate of $W_{g-1}(C)$ and so $-\psi(W)$ is a translate of $W_{g-d-1}(C)$ by Lemma \ref{lem:W=W_d}.
Hence, $W_{g-d}(C)+\psi(Y)$ is a translate of $W_{g-1}(C)$.
Applying Lemma \ref{lem:W=W_d} again shows then that $\psi(Y)$ is a translate of $W_{d-1}(C)$, as we want.

\textbf{Case 2:} $\psi(Z)$ is a translate of $-W_d(C)$. 

By (\ref{eq:psiZ-psiW}), $W_d(C)+\psi(W)$ is in this case a translate of $-W_{g-1}(C)$ and thus of $W_{g-1}(C)$.
By Lemma \ref{lem:W=W_d}, $\psi(W)$ is therefore a translate of $W_{g-d-1}(C)$.
Since $1\leq d\leq g-2$, it follows from (\ref{eq:psiZ-psiW}) that 
\begin{align} \label{eq2:thm:GV:W'}
W_{g-1}(C)=W_1(C)-W_1(C)+W' ,
\end{align}
where $W'$ is a translate of $\psi(Y)-W_{g-d-2}(C)$. 
By Lemma \ref{lem:W=W_d},
\begin{align} \label{eq:thm:GV:W'}
-W_1(C)+W'=W_{g-2}(C) .
\end{align}
Let $c_0\in C$ be the preimage of $0\in J(C)$ under the Abel--Jacobi embedding. 
Any point on $W'$ is then represented by a divisor $D-g\cdot c_0$ on $C$, where $D$ is effective of degree $ g$.
It follows from (\ref{eq:thm:GV:W'}) that $D-c_0 -c$ is effective for all $c\in C$. 
Thus, 
\[
D-c_0 \in W^1_{g-1}(C) \subseteq \Pic^{g-1}(C)
\]
is a divisor whose linear series is positive-dimensional. 
By (\ref{eq:thm:GV:W'}), we have $\dim(W')\geq g-3$ (in fact equality holds by Lemma \ref{lem:debarre}) and so $\dim(W^1_{g-1}(C))\geq g-3$.
A theorem of Martens \cite[p.\ 191]{arbarello-etal} implies that $C$ is hyperelliptic and so case 1 applies.
This concludes the proof. 
\end{proof}

\section{Dominations by products} \label{sec:DPC}

\subsection{The DPC Problem for theta divisors} \label{subsec:DPCTheta}
We have the following well-known

\begin{lemma} \label{lem:F:X1xX2}
Let $A$ be an abelian variety and let $F:Z_1\times Z_2 \dashrightarrow A$ be a rational map from a product of smooth varieties $Z_1$ and $Z_2$.
Then there are morphisms $f_i:Z_i\longrightarrow A$ for $i=1,2$ such that $F=f_1+f_2$.
\end{lemma}

\begin{proof}
Since $A$ does not contain rational curves, $F$ is in fact a morphism, which by the universal property of Albanese varieties factors through
$
\Alb(Z_1)\times \Alb(Z_2)
$. 
We conclude as morphisms between abelian varieties are translates of homomorphisms.
\end{proof}

The following result shows that property (\ref{item:Y=Y'+C}) in Theorem \ref{thm:GV} is in fact a condition on the birational geometry of $Z$.

\begin{corollary} \label{cor:curvesummand}
An $n$-dimensional subvariety $Z$ of an abelian variety $A$ has a $d$-dimensional summand if and only if there is a dominant rational map $F:Z_1\times Z_2\dashrightarrow Z$, where $Z_1$ and $Z_2$ are varieties of dimension $d$ and $n-d$ respectively.
\end{corollary}

\begin{proof}
If $Z$ has a $d$-dimensional summand $Z_1$, the decomposition $Z=Z_1+Z_2$ for a suitable $Z_2$ gives rise to a dominant rational map $F:Z_1\times Z_2\dashrightarrow Z$ as we want.
Conversely, if $F:Z_1\times Z_2\dashrightarrow Z$ is given, after resolving the singularities of $Z_1$ and $Z_2$, the assertion follows from Lemma \ref{lem:F:X1xX2}.
This proves Corollary \ref{cor:curvesummand}.
\end{proof}

Corollary \ref{cor:ThetaDPC} stated in the introduction is an immediate consequence of Riemann's Theorem and 

\begin{corollary} \label{cor:CxY}
Let $(A,\Theta)$ be an indecomposable $g$-dimensional ppav. 
Suppose there is a dominant rational map
\[
F:Z_1\times Z_{2} \dashrightarrow \Theta ,
\]
where $Z_1$ and $Z_2$ are varieties of dimension $1$ and $g-2$ respectively. 
Then $(A,\Theta)$ is isomorphic to the Jacobian of a smooth curve $C$. 
Moreover, if we identify $\Theta$ with $W_{g-1}(C)$, there are rational maps $f_1:Z_1\dashrightarrow W_1(C)$ and $f_2:Z_2\dashrightarrow W_{g-2}(C)$ with $F=f_1+f_2$.
\end{corollary}
\begin{proof}
After resolving the singularities of $Z_1$ and $Z_2$, we may assume that both varieties are smooth.
By Lemma \ref{lem:F:X1xX2}, $F:Z_1\times Z_2 \dashrightarrow \Theta\subseteq A$ is then a sum of morphisms $f_1:Z_1\longrightarrow A$ and $f_2:Z_2\longrightarrow A$.
Hence, 
\[
f_1(Z_1)+f_2(Z_2)=\Theta ,
\] 
and so Corollary \ref{cor:CxY} follows from Theorem \ref{thm:Theta}.
\end{proof}

\begin{remark}
For an arbitrary ppav $(A,\Theta)$, Corollary \ref{cor:ThetaDPC} implies that each component of $\Theta$ is DPC if and only if $(A,\Theta)$ is a product of Jacobians of smooth curves.
Indeed, if $(A,\Theta)=(A_1,\Theta_1)\times\dots \times(A_r,\Theta_r)$ with indecomposable factors $(A_i,\Theta_i)$, then $\Theta$ has $r$ components which are isomorphic to $\Theta_i\times \prod_{j\neq i}A_j$ where $i=1,\dots ,r$.
A product of varieties is DPC if and only if each factor is DPC.
Since abelian varieties are DPC, it follows that the components of $\Theta$ are DPC if and only if each $\Theta_i$ is DPC, hence the result by Corollary \ref{cor:ThetaDPC}. 
\end{remark}

\begin{corollary} \label{cor:Fano}
The Fano surface of lines on a smooth cubic threefold $X\subseteq \CP^4$ is not dominated by a product of curves.
\end{corollary}
\begin{proof}
By \cite[Thm.\ 13.4.]{clemens-griffiths}, the theta divisor of the intermediate Jacobian $(J^3(X),\Theta)$ is dominated by the product $S\times S$, where $S$ is the Fano surface of lines on $X$.
Since $(J^3(X),\Theta)$ is indecomposable and not isomorphic to the Jacobian of a smooth curve \cite[p.\ 350]{clemens-griffiths}, Corollary \ref{cor:Fano} follows from Corollary \ref{cor:CxY}.  
\end{proof}

\subsection{Dominations of symmetric products of curves} \label{subsec:martens}
Theorem \ref{thm:Theta} is nontrivial even in the case where $(A,\Theta)$ is known to be a Jacobian.
This allows us to classify all possible ways in which the symmetric product $C^{(k)}$ of a smooth curve $C$ of genus $g\geq k+1$ can be dominated by a product of curves.
Before we explain the result, we should note that 
$\AJ_k:C^{(k)}\longrightarrow W_k(C)$ is a birational morphism for $g\geq k$, and that $-W_{g-1}(C)$ is a translate of $W_{g-1}(C)$.
In particular, multiplication by $-1$ on $J(C)$ induces a nontrivial birational automorphism 
\[
\iota:C^{(g-1)}\stackrel{\sim}\dashrightarrow C^{(g-1)} .
\] 

\begin{corollary} \label{cor:C^(k)}
Let $C$ be a smooth curve of genus $g$.
Suppose that for some $k\leq g-1$, there are smooth curves $C_1,\dots ,C_{k}$ together with a dominant rational map 
\[
F:C_1\times\dots \times C_{k} \dashrightarrow C^{(k)} .
\]
Then there are dominant morphisms $f_i:C_i\longrightarrow C$ with the following property:
\begin{itemize}
	\item If $k<g-1$, then $F=f_1+\dots +f_k$.
	\item If $k=g-1$, then $F=f_1+\dots +f_{g-1}$ or $F=\iota\circ \left(f_1+\dots +f_{g-1}\right)$.
\end{itemize}
\end{corollary}

\begin{proof} 
We use the birational morphism $\AJ_k:C^{(k)}\longrightarrow W_k(C)$ to identify $C^{(k)}$ birationally with its image $W_k(C)$ in $J(C)$.
By Lemma \ref{lem:F:X1xX2}, the rational map
\[
\AJ_k \circ F: C_1\times\dots \times C_{k} \dashrightarrow W_k(C)
\] 
is a sum of morphisms $C_i\longrightarrow W_k(C)$. 
If $C'_i$ denotes the image of $C_i$ in $J(C)$, then
\begin{align} \label{eq:sumC_i'}
\Theta_C=C_1'+\dots + C_k'+W_{g-k-1}(C)
\end{align}
by Riemann's Theorem.
Proposition \ref{prop:a=1} yields therefore 
$
[C'_i]=\frac{1}{(g-1)!}[\Theta_C]^{g-1}
$ 
for all $i$. 
It follows for instance from Debarre's Theorem \cite{debarre} that each $C'_i$ is a translate of $C$ or of $-C$, where $C\subseteq J(C)$ is identified with its Abel--Jacobi image.
If $C$ is hyperelliptic, Corollary \ref{cor:C^(k)} follows.
 
Assume now that $C$ is non-hyperelliptic.
Then there is some $0\leq r\leq k$, such that $C_i$ is a translate of $-C$ for precisely $r$ many indices $i\in \left\{1,\dots ,k\right\}$. 
By (\ref{eq:sumC_i'}), $W_{g-r-1}(C)-W_r(C)$ is then a translate of $\Theta_C$.
However, Lemma 5.5 in \cite{debarre} yields
\[
[W_{g-r-1}(C)-W_r(C)]=\binom{g-1}{r}\cdot [\Theta_C] ,
\]
which coincides with $[\Theta_C]$ if and only if $r=0$ or $r=g-1$.
This proves Corollary \ref{cor:C^(k)}.
\end{proof}

Corollary \ref{cor:C^(k)} implies a theorem of Martens \cite{martens,ran3} asserting that any birational map
\[
C_1^{(k)}\stackrel{\sim}\dashrightarrow C_2^{(k)}
\] 
between the $k$-th symmetric products of smooth curves $C_1$ and $C_2$ of genus $g \geq k+2$ is induced by an isomorphism $C_1\stackrel{\sim}\longrightarrow C_2$.

For $k\geq g$, the symmetric product $C^{(k)}$ is birational to $J(C)\times \CP^{k-g}$.
This shows that Corollary \ref{cor:C^(k)} is sharp as for $k\geq g$, the product $J(C)\times \CP^{k-g}$ admits a lot of nontrivial dominations.
For instance, it is dominated by $k-g$ arbitrary curves (whose product dominates $\CP^{k-g}$) together with any choice of $g$ curves in $J(C)$ whose sum is $J(C)$.

\section*{Acknowledgment}
I would like to thank my advisor D.\ Huybrechts for constant support, encouragement and discussions about the DPC problem. 
Thanks go also to C.\ Schnell for his lectures on generic vanishing theory, held in Bonn during the winter semester 2013/14, where I learned about GV-sheaves and Ein--Lazarsfeld's result \cite{ein-laz}.
I am grateful to J.\ Fresan, D.\ Kotschick, L.\ Lombardi and M.\ Popa for useful comments.
Special thanks to the anonymous referee for helpful comments and corrections.
The author is member of the BIGS and the SFB/TR 45 and supported by an IMPRS Scholarship of the Max Planck Society.


\begin{thebibliography}{9}   

\bibitem{arbarello-etal}
E.\ Arbarello, M.\ Cornalba ,P.\ A.\ Griffiths and J.\ Harris, {\em Geometry of algebraic curves I},
Springer-Verlag, New York, 1985.

\bibitem{birkenhake-lange}
C.\ Birkenhake and H.\ Lange, {\em Complex abelian varieties}, 2nd edition, Springer--Verlag, 2004.

\bibitem{clemens-griffiths}
C.\ H.\ Clemens and P.\ A.\ Griffiths, {\em The intermediate Jacobian of the cubic threefold}, Ann.\ Math.\ \textbf{95} (1972), 281--356.


\bibitem{debarre}
O.\ Debarre, {\em Minimal cohomology classes and Jacobians}, J.\ Alg.\ Geom.\ \textbf{4} (1995), no. 2, 321--335.

\bibitem{debarre:book}
O.\ Debarre, {\em Tores et variétés abéliennes complex}, Cours Spécialisés 6, Société Mathématique de France, EDP Sciences, 1999.

\bibitem{deligne}
P.\ Deligne, {\em La conjecture de Weil pour les surfaces K3}, Invent.\ Math.\ \textbf{15} (1972), 206--226.

\bibitem{ein-laz}
L.\ Ein and R.\ Lazarsfeld, {\em Singularities of theta divisors and the birational geometry of irregular varieties}, J.\ Amer.\ Math.\ Soc.\ \textbf{10} (1997), 243--258. 


\bibitem{gunning}
R.\ C.\ Gunning, {\em Some curves in abelian varieties}, Invent.\ Math.\ \textbf{66} (1982), 377--389.

\bibitem{grushevsky}
S.\ Grushevsky, {\em The Schottky problem}, Current Developments in Algebraic Geometry, MSRI Publications \textbf{59}, Cambridge Univ. Press (2012), 129--164.




\bibitem{hoyt}
W.\ L.\ Hoyt, {\em On products and algebraic families of Jacobian varieties}, Ann.\ of Math.\ \textbf{77} (1963), 415--423.

\bibitem{huy}
D.\ Huybrechts, {\em Fourier-Mukai transforms in Algebraic Geometry}, Oxford Mathematical Monographs, Oxford, 2006. 

\bibitem{krichever}
I.\ Krichever, {\em Characterizing Jacobians via trisecants of the Kummer Variety}, Ann.\ of Math.\ \textbf{172} (2010), 485--516. 

\bibitem{little3}
J.\ Little, {\em Correction to: On Lie's approach to the study of translation manifolds}, on his personal webpage: http://mathcs.holycross.edu/~little/Corrs.html.

\bibitem{little2}
J.\ Little, {\em On Lie's approach to the study of translation manifolds}, J.\ Diff.\ Geom.\ \textbf{26} (1987), 253--272.


\bibitem{matsusaka}
T.\ Matsusaka, {\em On a characterization of a Jacobian variety}, Mem.\ Coll.\ Sci.\ Kyoto Ser.\ A Math.\ \textbf{32} (1959), 1--19. 

\bibitem{martens}
H.\ H.\ Martens, {\em An extended Torelli Theorem}, Amer.\ J.\ Math.\ \textbf{87} (1965), 257--261. 


\bibitem{pareschi-popa}
G.\ Pareschi and M.\ Popa, {\em Generic vanishing and minimal cohomology classes on abelian varieties}, Math.\ Ann.\ \textbf{340} (2008), 209--222.

\bibitem{ran2}
Z.\ Ran, {\em A characterization of five-dimensional Jacobian varieties}, Invent.\ Math.\ \textbf{67} (1982), 395--422.

\bibitem{ran3}
Z.\ Ran, {\em On a theorem of Martens}, Rend.\ Sem.\ Mat.\ Univers.\ Politecn.\ Torino \textbf{44} (1986), 287--291.

\bibitem{ran}
Z.\ Ran, {\em On subvarieties of abelian varieties}, Invent.\ Math.\ \textbf{62} (1981), 459--479.


\bibitem{schoen}
C.\ Schoen, {\em Varieties dominated by product varieties}, Internat.\ J.\ Math.\ \textbf{7} (1996), 541--571.

\bibitem{grothendieck-serre}
J.-P. Serre, {\em Letter to Grothendieck, March 31, 1964}, in: 
{\em Grothendieck-Serre correspondence}, AMS, Providence R.I.\ (2004). 

\bibitem{weil}
A.\ Weil, {\em Zum Beweis des Torellischen Satzes}, Nachr.\ Akad.\ Wiss.\ Göttingen (1957), 33--53.

\bibitem{welters}
G.\ E.\ Welters, {\em A characterization of non-hyperelliptic Jacobi varieties}, Invent.\ Math.\ \textbf{74} (1983), 437--440.


\end{thebibliography}
\end{document}